\documentclass[12pt]{amsart}
\usepackage[hmargin=2.5cm,vmargin=2.5cm]{geometry}
\usepackage{amsfonts, amstext, amsmath, amsthm, amscd, amssymb}
\usepackage[pdftex]{graphicx, color}

\usepackage[noadjust]{cite}
\usepackage[hidelinks,pagebackref]{hyperref}
\usepackage{enumitem}
\usepackage{setspace}
\usepackage{float}
\usepackage{subcaption}
\usepackage{mathrsfs}
\usepackage{tikz}

\AtBeginDocument{
	\def\MR#1{}
}

\pagecolor{white}

\allowdisplaybreaks

\newcommand{\Z}{\mathbb{Z}}

\newcommand{\Q}{\mathbb{Q}}

\newtheorem{thm}{Theorem}
\numberwithin{thm}{section}

\newtheorem{conj}[thm]{Conjecture}
\newtheorem{prop}[thm]{Proposition}
\newtheorem{lemma}[thm]{Lemma}
\newtheorem{cor}[thm]{Corollary}

\newtheorem{conv}[thm]{Convention}

\newtheorem*{namedtheorem}{\theoremname}
\newcommand{\theoremname}{testing}

\theoremstyle{definition}

\newtheorem*{nameddef}{\defname}
\newcommand{\defname}{testing}

\theoremstyle{remark}

\begin{document}
	\title{On the Cosmetic Crossing Conjecture for Special Alternating Links}
	\author{Joe Boninger}
	\address{Department of Mathematics, Boston College, Chestnut Hill, MA}
	\email{boninger@bc.edu}
	\maketitle
	
	\begin{abstract}
		We prove that a family of links, which includes all special alternating knots, does not admit non-nugatory crossing changes which preserve the isotopy type of the link. Our proof incorporates a result of Lidman and Moore on crossing changes to knots with $L$-space branched double-covers, as well as tools from Scharlemann and Thompson's proof of the cosmetic crossing conjecture for the unknot.
	\end{abstract}
	
	\section{Introduction}
	
	The cosmetic crossing conjecture, attributed to Xiao-Song Lin \cite[Problem 1.58]{kir78}, posits that changing a nontrivial crossing in a link diagram must change the isotopy type of the link. More concretely, given an oriented link $L \subset S^3$, define a \emph{crossing disk} to be a disk $D \subset S^3$ which intersects $L$ transversely at two points of opposite orientation. A \emph{crossing change} is then performed by passing a neighborhood of one point of $L \cap D$ through a neighborhood of the other, as in Figure \ref{fig:cc}. The crossing is said to be \emph{nugatory} if $\partial D$ bounds a disk in $S^3 - L$, and a crossing change is \emph{cosmetic} if it preserves the isotopy type of $L$.
	
	\begin{conj}[Cosmetic Crossing Conjecture]
		\label{thm:ccc}
		For any knot $L \subset S^3$, only a nugatory crossing admits a cosmetic crossing change.
	\end{conj}

	Conjecture \ref{thm:ccc} has been affirmed for two-bridge knots \cite{tor99} and fibered knots \cite{kal12}, and significant partial results exist for genus one knots and satellite knots \cite{ito22, ito222, bk16, bfkp12}. Further, Lidman and Moore have verified the conjecture for all knots $L \subset S^3$ such that the branched double-cover $\Sigma(L)$ is an $L$-space, and $L$ has square-free determinant \cite{lm17}; their work has been extended by Ito \cite{ito223}.
	
	In this note, we prove the cosmetic crossing conjecture for all special alternating knots in $S^3$. (The case of special alternating knots with square-free determinant is included in \cite{lm17}.) 
	\begin{thm}
		\label{thm:spec-alt}
		Let $L \subset S^3$ be a special alternating knot. Then $L$ admits no cosmetic, non-nugatory crossing change.
	\end{thm}
	Actually, we prove Conjecture \ref{thm:ccc} for a family of oriented links which includes all non-split special alternating links with certain orientations, and some non-alternating links---see Theorem \ref{thm:main-one} below.

	A diagram $D \subset S^2$ of a link $L \subset S^3$ is \emph{alternating} if crossings alternate over-under-~as one traverses any link component of the diagram. The diagram is \emph{special} if one of its checkerboard surfaces, constructed by shading the components of $S^2 - D$ in a checkerboard fashion and taking the union of the shaded regions with half-twisted bands at each crossing, is orientable. Equivalently, a diagram is special if one of its Tait graphs is bipartite. A link $L \subset S^3$ is called \emph{special alternating} if it admits a diagram which is both alternating and special. Special alternating links include $(2,n)$-torus links, and many twist and pretzel knots. More generally, as alluded to above, a special alternating diagram can be constructed from any embedding of a bipartite planar graph in $S^2$.
	
	Our proof of Theorem \ref{thm:spec-alt} incorporates a key result from Lidman and Moore \cite{lm17}, as well as tools from Scharlemann and Thompson's proof of Conjecture \ref{thm:ccc} for the unknot \cite[Theorem 1.4]{st89}. As a corollary, we obtain the following:
	
	\begin{cor}
		\label{thm:seifert}
		Suppose a link $L \subset S^3$ admits a cosmetic, non-nugatory crossing change, and $\Sigma(L)$ is an $L$-space. Then $L$ bounds two minimal-genus Seifert surfaces, with Seifert forms represented by matrices $(v_{ij})$ and $(v'_{ij})$, such that $v_{11} = v'_{11} + 1$ and $v_{ij} = v'_{ij}$ otherwise.
	\end{cor}

	Corollary \ref{thm:seifert} is analogous to a finding of Balm, Friedl, Kalfagianni and Powell \cite[Corollary 1.3]{bfkp12}, who use a related approach to study genus one knots.

	\subsection{Acknowledgements}
	
	The author thanks Jacob Caudell for introducing him to the cosmetic crossing conjecture, Josh Greene for helpful conversations, and an anonymous reviewer for insightful feedback and corrections. This material is based upon work supported by the National Science Foundation under Award No.~2202704.
	
	\begin{figure}[H]
		\centering
		\includegraphics[height=3cm]{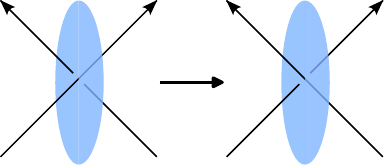}
		\caption{A crossing change}
		\label{fig:cc}
	\end{figure}
	
	\section{Background}
	
	A three-manifold $Y$ is an \emph{$L$-space} if it is a rational homology sphere with rank$(\widehat{HF}(Y)) = |H_1(Y;\Z)|$, where $\widehat{HF}$ denotes the hat flavor of Heegaard Floer homology. Of importance to us is the fact that, if $L \subset S^3$ is an alternating link, then its branched double-cover, $\Sigma(L)$, is an $L$-space \cite{os05}.
	
	Let $L \subset S^3$, and $D$ a crossing disk for $L$ as above. A \emph{crossing arc} is an embedded arc $\gamma \subset D$ connecting the two points of $L \cap D$, and we use $\tilde{\gamma}$ to denote the closed curve which is the preimage of $\gamma$ in the branched covering $\Sigma(L) \to S^3$. Lidman and Moore proved the following:
	
	\begin{thm}[\protect{\cite[Remark 13]{lm17}}]
		\label{thm:lm}
	Let $L$ be an oriented knot with $\Sigma(L)$ an $L$-space, $D$ a crossing disk for $L$, and $\gamma$ a crossing arc in $D$. If the crossing change induced by $D$ is cosmetic, and $\tilde{\gamma}$ is nullhomologous in $\Sigma(L)$, then $D$ is nugatory.
	\end{thm}
	
	Their argument uses the surgery characterization of an unknot in an $L$-space, due to Gainullin \cite{gai18}. In the appendix, we extend Theorem \ref{thm:lm} to links.
	
	Next, we recall the \emph{Gordon-Litherland form}. Given a surface $S \subset S^3$, this is a symmetric, bilinear form $\mathcal{G}_S : H_1(S)^2 \to \Z$ \cite{gl78}. Briefly, let $\nu(S) \subset S^3$ denote the unit normal bundle of $S$, with projection $p : \nu(S) \to S$. Given homology classes $a, b \in H_1(S)$, represented by embedded multi-curves $\alpha, \beta \subset S$, we define
	$$
	\mathcal{G}_S(a,b) = \text{lk}(\alpha, p^{-1} \beta),
	$$
	where lk is the linking number. If $L \subset S^3$ is an oriented link, and $S$ a compatibly oriented Seifert surface for $L$, then $\mathcal{G}_S$ coincides with the symmetrized Seifert form of $S$, and the signature $\sigma(\mathcal{G}_S)$ equals the signature of $L$. If, in addition, $S$ is connected, then the nullity $\eta(\mathcal{G}_S)$ is a link invariant called the \emph{nullity} of $L$, $\eta(L)$. (In some literature, $\eta(L)$ is defined to be $\eta(\mathcal{G}_S) + 1$.)
	
	\begin{conv}
		All links are oriented, and we require Seifert surfaces be oriented compatibly with the link. We allow Seifert surfaces to be disconnected, but not to have closed components.
	\end{conv}

	A surface in $S^3$ is called \emph{definite} if its Gordon-Litherland form is positive- or negative-definite. If $D \subset S^2$ is an alternating link diagram, then the two checkerboard surfaces of $D$ are known to be definite; conversely, definite surfaces can be used to characterize alternating links topologically \cite{g17, h17}. In particular, a suitably oriented special alternating link bounds a definite Seifert surface.
	
	\section{Proof of Main Result}
	
	We say a Seifert surface spanning an oriented, non-split link $L \subset S^3$ is \emph{taut} if it has maximal Euler characteristic among all Seifert surfaces of $L$. (For equivalence with the standard definition of tautness, see \cite[Lemma 1.2]{st89}.) We have:
	
	\begin{lemma}
		\label{thm:taut_def}
		Suppose non-split $L \subset S^3$ bounds a definite Seifert surface $S$. Then $S$ is taut in $S^3 - L$, and conversely every taut Seifert surface for $L$ is definite.
	\end{lemma}

	\begin{proof}
		First, we argue that $S$ has the maximal number of components of any Seifert surface for $L$. Suppose some Seifert surface $S'$ has $b_0(S') > b_0(S)$. We form a connected Seifert surface $\hat{S}$ for $L$ by joining the components of $S$ using $b_0(S) - 1$ tubes, and likewise form a connected surface $\hat{S}'$ by adding $b_0(S') - 1$ tubes to $S'$. We have
			$$
			\eta(\mathcal{G}_{\hat{S}}) = \eta(\mathcal{G}_{\hat{S'}}) \geq b_0(S') - 1,
			$$
		since each tube increases the nullity by one. It follows that
			$$
			\eta(\mathcal{G}_S) = \eta(\mathcal{G}_{\hat{S}}) - b_0(S) + 1 \geq b_0(S') - b_0(S) > 0,
			$$
		contradicting the definite-ness of $S$.
			
		Next, as in \cite[Proposition 3.1]{g17}, for any Seifert surface $S'$ of $L$, we have
		$$
		b_1(S') \geq |\sigma(L)| = b_1(S),
		$$
		the last equality following from the fact that $S$ is definite. This shows $S$ has minimal $b_1$, and therefore maximal Euler characteristic. Finally, any Seifert surface $S'$ with $\chi(S') = \chi(S)$ must have $b_1(S') = b_1(S) = |\sigma(L)|$, so must be definite as well.
	\end{proof}
	
	\begin{thm}
		\label{thm:main-one}
		Suppose an oriented link $L \subset S^3$ satisfies the following conditions:
			\begin{itemize}
				\item The link $L$ bounds a definite Seifert surface $S$.
				\item The branched double-cover $\Sigma(L)$ is an $L$-space.
			\end{itemize}
		Then $L$ does not admit a non-nugatory, cosmetic crossing change.
	\end{thm}

	We note the second condition above implies $L$ is non-split, since $\Sigma(L)$ is a rational homology sphere. Examples of non-alternating links which satisfy the hypotheses of Theorem \ref{thm:main-one} include the knots $9_{49}$, $10_{134}$, and $10_{142}$. These knots are known to be quasi-alternating \cite{man07, ck09-2}, and hence have branched double-covers which are $L$-spaces. Further, each knot $K$ satisfies $2g(K) = |\sigma(K)|$, $g$ the genus of $K$, implying the existence of a definite Seifert surface. These examples were found with the help of KnotInfo \cite{knotinfo}.

	\begin{proof}[Proof of Theorem \protect{\ref{thm:main-one}}.]
	Let $L$ be a link satisfying the hypotheses of the theorem, and let $D$ be a cosmetic crossing disk for $L$. Let $K = \partial D$, and let $M = S^3 - n(K \cup L)$, where $n$ indicates a regular neighborhood. Following \cite{st89}, let $M_{-1}$, $M_0$, and $M_\infty$ denote the result of filling $M$ along $\partial n(K)$ by a solid torus with slope $-1$, $0$, and $\infty$ respectively. Then $M_\infty = S^3 - n(L)$, and without loss of generality, $M_{-1}$ is the result of performing the crossing change indicated by $D$. By assumption, $M_{-1} \cong M_{\infty}$.
	
	Let $S \subset M$ be a Seifert surface for $L$ which is taut in $M$. Shrinking $D$ if necessary, we may assume that $S \cap D$ is a single arc $\gamma$, which is also a crossing arc for $D$. Scharlemann and Thompson prove that $S$ is taut in at least two of $M_{-1}$, $M_0$, and $M_\infty$ \cite[Claim 1]{st89}. Thus $S$ is taut in at least one of $M_{-1}$ and $M_\infty$, and since these manifolds are homeomorphic, $S$ is taut in both. Let $S$ denote the inclusion of $S$ in $M_\infty$, and let $S'$ denote the inclusion of $S$ in $M_{-1}$. It follows from Lemma \ref{thm:taut_def} that both $S$ and $S'$ are definite.
	
	We consider two cases.
	
	\textbf{Case 1: The arc $\gamma$ separates $S$.} Let $S''$ be one of the components of $S - \gamma$, and let $\tilde{S}, \tilde{S''}, \tilde{L}, \tilde{\gamma} \subset \Sigma(L)$ denote the respective preimages of $S$, $S''$, $L$, and $\gamma$ in the branched covering $\Sigma(L) \to S^3$. (Here we view $S$ as a subset of $S^3$, rather than a subset of $S^3 - n(L)$.) Considering the classical construction of a branched cover from a Seifert surface \cite{rol76}, we see that $\tilde{S} - n(\tilde{L})$ consists of two lifted copies of $S - n(L)$; we orient these copies by lifting an orientation from $S - n(L)$. When restricted to a meridian circle of $\partial n(\tilde{L})$, the covering map $\Sigma(L) \to S^3$ has the form $z \mapsto z^2$. Thus, near such a meridian, the two components of $\tilde{S} - n(\tilde{L})$ are oriented as in Figure \ref{fig:ls}.
	
	The surface $\tilde{S}$ is constructed by gluing the two lifted copies of $S - n(L)$ together along the annuli $\tilde{S} \cap n(\tilde{L})$. With Figure \ref{fig:ls} in mind, by switching the orientation of one of the lifted copies, these annuli can be made to preserve orientation, and therefore $\tilde{S}$ is orientable. Since $\tilde{S}'' \subset \tilde{S}$, $\tilde{S}''$ is also orientable, and its boundary is exactly $\tilde{\gamma}$. The existence of $\tilde{S}''$ shows $\tilde{\gamma}$ is nullhomologous in $H_1(\Sigma(L))$, so Theorem \ref{thm:lm} implies the crossing change is nugatory in this case.
	
	\begin{figure}[H]
		\centering
		\includegraphics[height=1cm]{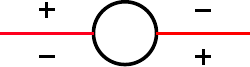}
		\caption{Two oriented lifts of $S - n(L)$, in a neighborhood of a meridian of $\partial n(\tilde{L})$}
		\label{fig:ls}
	\end{figure}

	\textbf{Case 2: The arc $\gamma$ does not separate $S$.} In this case, we choose a basis $a_1, \dots, a_n$ for $H_1(S)$, represented by curves $\alpha_1, \dots, \alpha_n \subset S$ respectively, such that $\alpha_1$ intersects $D$ one time, and $\alpha_i \cap D = \varnothing$ for $i \neq 1$. Let $G = (g_{ij})$ be the symmetric matrix representing the Gordon-Litherland form $\mathcal{G}_S$ in this basis. We also let $a_1, \dots, a_n$ denote the same basis for $H_1(S')$, i.e.~the basis induced by the inclusion $S \subset M \hookrightarrow M_{-1}$. Let $G' = (g'_{ij})$ be the corresponding matrix representing $\mathcal{G}_{S'}$.
	
	We have $|\det(G)| = |\det(G')| = \det(L)$, and since $\mathcal{G}_S$ and $\mathcal{G}_{S'}$ are both definite of the same rank and sign, determined by $\sigma(L)$, $\det(G) = \det(G')$. Further, by inspecting how $S$ changes in a neighborhood of $D$ when $(-1)$-surgery is performed, we calculate that $g'_{11} + 2 = g_{11}$, and $g_{ij} = g'_{ij}$ for $i$ and $j$ not both equal to one. We consider computing the determinants of $G$ and $G'$ using a Laplace expansion along the top row---since the two quantities are equal, and the matrices differ at only one entry, we find
	$$
	g_{11}\det(G_{11}) = g'_{11}\det(G'_{11}) = (g_{11} + 2)\det(G_{11}),
	$$
	where $G_{11}$ denotes the matrix formed by removing the first row and column of $G$. This matrix represents the restriction of $\mathcal{G}_S$ to the subspace of $H_1(S)$ spanned by $a_2, \dots, a_n$; as the restriction of a definite form, this form is also definite, and hence $\det(G_{11}) \neq 0$. We conclude that
	$$
	g_{11} = g_{11} + 2,
	$$
	a contradiction which indicates this case cannot occur.
	\end{proof}

	\begin{proof}[Proof of Corollary \protect{\ref{thm:seifert}}]
		Following the proof of Theorem \ref{thm:main-one}, we obtain two taut Seifert surfaces for $L$, with the crossing arc $\gamma$ embedded as a non-separating arc in each. Choosing the homology bases $a_1, \dots, a_n$, as above, gives the desired Seifert matrices.
	\end{proof}

	Finally, we give a minor application of Corollary \ref{thm:seifert}.
	
	\begin{cor}
		\label{thm:minor}
		Suppose a knot $L \subset S^3$ admits a cosmetic, non-nugatory crossing change, and $\Sigma(L)$ is an $L$-space. Then, letting $m$ denote the size of a minimal generating set for $H_1(\Sigma(L))$, we have $m < 2g(L)$.
	\end{cor}

	\begin{proof}
		Let $G$ and $G'$ be the two matrices obtained in the proof of Theorem \ref{thm:main-one}, representing two Gordon-Litherland forms of $L$ with rank $2g(L)$. We use the fact that $G$ and $G'$ give presentations for the finite abelian group $H_1(\Sigma(L))$, and compute this group's invariant factors. For an invertible matrix $A$, let $\Gamma^A_i$ denote the greatest common divisor of the determinants of the $i$-by-$i$ minors of $A$, and let $\delta^A_i = \Gamma^A_i/\Gamma^A_{i - 1}$. We recall, via the Smith normal form of $A$, that the invariant factors of the abelian group presented by $A$ are given by the set of all $\delta^A_i$ not equal to $1$.
		
		Since $G$ and $G'$ have the same rank and present the same group, we have
		$$
		\gcd_{ij}(g_{ij}) = \delta^G_1 = \delta^{G'}_1 = \gcd_{ij}(g'_{ij}).
		$$
		Because $g_{11} = g'_{11} + 2$, $\delta^G_1$ divides $2$. Additionally, since $\prod_i \delta^G_i = \det(L)$, and knots have odd determinant, we have $\delta^G_1 = 1$. Thus
		$m < \text{rk}(G) = 2g$, as desired.
	\end{proof}

	This result extends \cite[Theorem 1.1(2)]{bfkp12}. In general $m \leq 2g(L)$, but equality is occasionally attained. For example, the pretzel knot $K = P(9,9,9,9,-27)$ is quasi-alternating by \cite[Theorem 3.2(1)]{ck09-2}, hence has branched double-cover an $L$-space. The knot $K$ has genus two and $H_1(\Sigma(K)) \cong \Z/9 \oplus \Z/9 \oplus \Z/9 \oplus \Z/99$, so Corollary \ref{thm:minor} shows $K$ does not admit cosmetic crossings. This example is easily generalized, for instance by considering the family of pretzel knots $P(m^2, m^2, m^2, m^2, -3m^2)$ with $m$ odd, to produce many new examples of knots which do not admit cosmetic crossings. Choosing square numbers ensures the resulting pretzel knot is not included in the main theorem of \cite{lm17}.
	
	\appendix
	
	\section{Extending Theorem \protect\ref{thm:lm} to Links}
	
	In what follows, let $L \subset S^3$ be a link, $D$ a crossing disk, and $\gamma$ the associated crossing arc. As above, let $\tilde{\gamma}$ denote the closed curve which is the preimage of $\gamma$ in the branched cover $\Sigma(L)$.
	
	The extension of Theorem \ref{thm:lm} to links ultimately reduces to the following proposition.
	
	\begin{prop}
		\label{thm:links_one}
		Suppose $\det(L) \neq 0$, and the crossing change associated with $D$ is cosmetic. If $\tilde{\gamma}$ is a null-homologous unknot in $\Sigma(L)$, then $D$ is nugatory.
	\end{prop}

	To complete the argument, the reader may consult the proof of \cite[Thm.~2]{lm17}, using Proposition \ref{thm:links_one} in place of \cite[Prop.~12]{lm17}. Our proof closely follows that of the latter proposition, and we set up some additional notation before sketching it. Let $B \subset S^3$ be a regular neighborhood of $\gamma$, chosen so that $B \cap D$ is a disk contained in int$(D)$, and so that $B \cap L$ consists of two arcs. Observe that the preimage $\tilde B \subset \Sigma(L)$ of $B$ under the branched covering is a solid torus, and let $N = \Sigma(L) - \tilde B$. Since $\det(L) \neq 0$, $\Sigma(L)$ is a rational homology sphere, and a Mayer-Vietoris argument shows $H_2(N;\Q) \cong 0$ and $H_1(N;\Q) \cong \Q$. There is a unique slope $\lambda_N$ of $\partial N$ which generates the kernel of the inclusion-induced map $H_1(\partial N; \Q) \to H_1(N; \Q)$. This slope $\lambda_N$ is called the {\em rational longitude} of $N$; we refer the reader to \cite{lm17, wat12} for more details.

	\begin{proof}
		Let $\tilde\Gamma \subset \Sigma(L)$ be a disk with boundary $\tilde\gamma$; by definition, $\tilde\Gamma \cap \partial N$ is the rational longitude $\lambda_N$ of $N$. Let $\tau$ denote the covering involution on $\Sigma(L)$. By the equivariant Dehn's Lemma, we may assume that either $\tau(\tilde\Gamma) \cap \tilde\Gamma = \varnothing$ or $\tau(\tilde\Gamma) = \tilde\Gamma$.
		
		Suppose $\tau(\tilde\Gamma) \cap \tilde\Gamma$ is empty. This implies $\tilde\Gamma$ descends to a properly embedded disk $\Gamma$ in $S^3 - B$. Since $\tilde\Gamma$ avoids the fixed-point set of $\tau$, which is the preimage of $L$, the disk $\Gamma$ is disjoint from $L$. To show $D$ is nugatory, we will show that $\partial \Gamma$ is parallel to $D \cap \partial B$ in $\partial B - L$. If follows that $\partial D$ bounds a disk disjoint from $L$, formed by gluing $\Gamma$ to the annulus $D - B$. To show $\partial \Gamma$ and $D \cap \partial B$ are parallel in $\partial B - L$, it suffices to show that $D \cap \partial B$ lifts to $\lambda_N$ in $\partial N$.
		
		Let $L_0$ be the link formed by replacing the crossing ball $B$ with the ball shown in Figure \ref{fig:zero}, which we label $B_0$. Let $\Delta$ denote the Alexander polynomial, which satisfies the skein relation
		$$
		\Delta_{L_+} (x) - \Delta_{L_-}(x) = -(x^{-1/2} - x^{-1/2})\Delta_{L_0}(x).
		$$
		Since $L_+ = L_- = L$, we conclude $\Delta_{L_0} \equiv 0$. In particular, $\det(L_0) = \Delta_{L_0}(-1) = 0$, so $H_1(\Sigma(L_0))$ is infinite, and by Poincar\'e duality and the universal coefficient theorem, so is $H_2(\Sigma(L_0))$. Let $\tilde{B}_0$ be the preimage of $B_0$ in $\Sigma(L_0)$, which is equivalent to a Dehn filling of $N$ along some slope $\gamma_0$. Using the fact that $H_2(N;\Q) \cong 0$, the Myer-Vietoris theorem gives an exact sequence
		$$
		0 \to H_2(\Sigma(L_0);\Q) \to H_1(\partial N; \Q) \to H_1(N; \Q) \oplus H_1(\tilde{B}_0; \Q).
		$$
		Let $a \in H_2(\Sigma(L_0); \Q)$ be non-trivial, and let $\partial a$ be its (non-trivial) image in $H_1(\partial N; \Q)$. By exactness, $\partial a$ is in the kernel of the second map, so $\partial a$ is trivial in $H_1(\tilde{B}_0; \Q)$ and $H_1(N; \Q)$. Since $\partial a$ is trivial in $H_1(\tilde{B}_0; \Q)$, $\partial a$ is a rational multiple of $\gamma_0$ (forgetting the orientation of the former). Since $\partial a$ is trivial in $H_1(N; \Q)$, $\partial a$ is a rational multiple of $\lambda_N$. Thus $\gamma_0 = \lambda_N$.
		
		We've shown the rational longitude of $N$ corresponds to the slope $\gamma_0$ of the Dehn filling $\tilde{B}_0$. Since $D \cap B_0$ is a disk separating the two components of $L_0 \cap B_0$, $D \cap B_0$ lifts to a meridian disk of $\tilde{B}_0$, and $D \cap \partial B_0 = D \cap \partial B$ lifts to $\gamma_0 = \lambda_N$. This completes the proof in this case, and the case of $\tau(\tilde\Gamma) = \tilde\Gamma$ is handled just as in the proof of \cite[Prop.~12]{lm17}.
	\end{proof}

	\begin{figure}
		\centering
		\subcaptionbox{$L_+$}{
			\includegraphics[height=2cm]{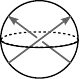}
		}
		\hspace{1cm}
		\subcaptionbox{$L_-$}{
			\includegraphics[height=2cm]{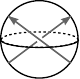}
		}
		\hspace{1cm}
		\subcaptionbox{$L_0$ \label{fig:zero}}{
			\includegraphics[height=2cm]{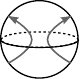}
		}
		\caption{}
		\label{fig:cballs}
	\end{figure}
	
	\bibliography{volume_conjecture}{}
	\bibliographystyle{amsplain}
	
\end{document}